\documentclass[reqno,11pt]{amsart}
\usepackage{amsthm,wrapfig,amsmath,amssymb,amsfonts,setspace,verbatim,graphicx,mathtools,mathrsfs,commath,float,mdframed,frame,xcolor,qtree,enumitem, ulem,afterpage}
\usepackage[hidelinks]{hyperref}
\usepackage[T1]{fontenc}

\DeclareMathOperator{\dom}{dom}

\DeclareMathOperator{\suc}{succ}
\DeclareMathOperator{\diam}{diam}

\usepackage{newunicodechar}
\usepackage{lmodern}
\usepackage[margin=1.39in]{geometry}
\newtheorem{ut}{Theorem}

\newtheorem{up}{Proposition}
\newtheorem{ul}[ut]{Lemma}

\newtheorem{uc}[ut]{Corollary}

\newtheorem{ue}{Example}
\newtheorem{uq}{Question}
\theoremstyle{remark}

\theoremstyle{definition}

\newtheorem{ur}{Remark}

\begin{document}

\title{A characterization of Erd\H{o}s space factors}
\author{David S. Lipham}

\address{Department of Mathematics, Auburn University at Montgomery, Montgomery 
AL 36117, United States of America}
\email{dlipham@aum.edu, dsl0003@auburn.edu}

\begin{abstract}We prove that an almost zero-dimensional space $X$ is an Erd\H{o}s space factor if and only if $X$ has a Sierpi\'{n}ski stratification of C-sets. We apply this characterization to spaces which are  countable unions of C-set Erd\H{o}s space factors. 
We show that the Erd\H{o}s space $\mathfrak E$ is unstable  by giving strongly $\sigma$-complete and nowhere $\sigma$-complete examples of almost zero-dimensional $F_{\sigma\delta}$-spaces which are not Erd\H{o}s space factors. This answers a question by Dijkstra and van Mill. \end{abstract}

\maketitle

\section{Introduction}

All spaces under consideration are non-empty, separable and metrizable.

We say that a subset $A$ of a space $X$ is  a \textbf{C-set} in $X$ if $A$ can be
written as an intersection of clopen subsets of $X$. A space $X$ is  \textbf{almost zero-dimensional} if every point $x\in X$ has a neighborhood basis consisting of C-sets  of $X$.  A (separable metric) topology $\mathfrak T$ on $X$  \textbf{witnesses the almost zero-dimensionality
of $X$} if $\mathfrak T$ is coarser than the given topology on $X$, $(X,\mathfrak T)$ is zero-dimensional, and every point of $X$ has a neighborhood basis  consisting of sets that are closed in $(X,\mathfrak T)$.  A space is almost zero-dimensional if and only if there is a topology witnessing the fact \cite[Remark 2.4]{erd}.  

Almost zero-dimensional spaces of positive dimension include the  \textbf{Erd\H{o}s space}  $$\mathfrak E:=\{x\in \ell^2:x_n\in \mathbb Q\text{ for all }n<\omega\}.$$ The almost zero-dimensionality of $\mathfrak E$ is witnessed by the $F_{\sigma\delta}$ topology that $\mathfrak E$ inherits from $\mathbb Q ^\omega$. We call a space $X$ an \textbf{Erd\H{o}s space factor} if there is a space $Y$ such that $X\times Y$ is homeomorphic to $\mathfrak E$.  

In 2010, Dijkstra and van Mill proved:

 \begin{up}[{\cite[Theorem 9.2]{erd}}]For any space $X$ the following are equivalent.
\begin{enumerate}[label=(\alph*)]
\item $X$ is an Erd\H{o}s space factor;
\item $X\times \mathfrak E$ is homeomorphic to $\mathfrak E$;
\item $X$ admits a closed embedding into $\mathfrak E$;
\item  there exists an $F_{\sigma\delta}$ topology witnessing the almost zero-dimensionality of $X$.
\end{enumerate}
\end{up}

\begin{up}[{\cite[Corollary 9.3]{erd}}]Every  almost zero-dimensional complete ($G_{\delta}$-)space is an Erd\H{o}s space factor.\end{up}

They then asked the following question, motivated by van Engelen's result  that $X\times \mathbb Q ^\omega\simeq \mathbb Q ^\omega$ for every  zero-dimensional $F_{\sigma\delta}$-space $X$ \cite{vee}. Note that  zero-dimensional $F_{\sigma\delta}$-spaces are also  Erd\H{o}s space factors by condition (d) in Proposition 1.

\begin{uq}[{\cite[Question 9.7]{erd}}]Is every almost zero-dimensional $F_{\sigma\delta}$-space an Erd\H{o}s space factor?\end{uq}

In this paper we give an intrinsic characterization of Erd\H{o}s space factors  (Theorem 1) and  answer Question 1 in the negative.  

Our  characterization will imply that if $X$ is an almost zero-dimensional space which can be written as a countable union of C-set Erd\H{o}s space factors, then $X$ is an Erd\H{o}s space factor (Theorem 2). For example, if $X$ is an Erd\H{o}s space factor then so is the Vietoris hyperspace $\mathcal F(X)$  (Corollary 4). Combining Theorem 2 and Proposition 2 will show that  every almost zero-dimensional  countable union of complete C-sets is an Erd\H{o}s space factor (Corollary 5). This result applies to Dijkstra's homogeneous  space $T(\tilde E,E')$ featured in  \cite{dij} (Corollary 6).  

We will present three counterexamples to Question 1. One  is nowhere $\sigma$-complete and $F_{\sigma\delta}$, similar to $\mathfrak E$. Another is strongly $\sigma$-complete and thus automatically $F_{\sigma\delta}$ \cite[Lemma 2.1]{vevm}. All three examples are  first category, yet they  have the property that   countable unions  of nowhere dense C-sets have empty interiors. An Erd\H{o}s space factor cannot have that combination of properties (Theorem 8). 

An element $X$ of a class of topological spaces is called the \textbf{stable space} for that class if for every space $Y$ in the class we have that $X\times Y$ is homeomorphic to $X$.  The negative answer to Question 1 shows that $\mathfrak E$ is unstable amongst the almost zero-dimensional $F_{\sigma\delta}$-spaces.  
Likewise, the \textbf{complete Erd\H{o}s space} $$\mathfrak E_{\mathrm{c}}:=\{x\in \ell^2:x_n\in \mathbb P\text{ for all }n<\omega\}$$ is unstable in the class of almost zero-dimensional Polish spaces because $\mathfrak E_{\mathrm{c}}$ is not homeomorphic to its $\omega$-power \cite{stab}. By contrast,  $\mathfrak E_{\mathrm{c}}^\omega$ is  stable and is therefore the almost zero-dimensional analogue of $\mathbb P$  (the space of irrational numbers) \cite{erd2}.  It is unknown whether a stable element exists for the class of almost zero-dimensional $F_{\sigma\delta}$-spaces.

\section{The characterization}

A \textbf{tree} $T$ on an alphabet $A$ is a subset of $A^{<\omega}$ that is closed under initial segments, i.e.\ if $\beta \in T$ and $\alpha\prec \beta$ then $\alpha\in T$.  An element $\lambda\in A^\omega$ is an \textbf{infinite branch} of $T$ provided $\lambda\restriction k \in T$ for every $k <\omega$. We let  $[T]$ denote  the set of all infinite branches of $T$. If $\alpha,\beta\in T$ are such that $\alpha\prec \beta$ and $\dom(\beta)=\dom(\alpha)+1$, then we say that $\beta$ is an \textbf{immediate successor} of $\alpha$  and $\suc(\alpha)$ denotes the set of immediate successors of $\alpha$ in $T$. 

A system $(X_\alpha)_{\alpha\in T}$ is a \textbf{Sierpi\'{n}ski stratification} of a space $X$ if:
\begin{enumerate}
 \item $T$ is a non-empty tree over a countable alphabet,

 \item each $X_\alpha$ is a closed subset of $X$,   
 
\item  $X_\varnothing = X$  and $X_\alpha=\bigcup \{X_\beta:\beta\in \suc(\alpha)\}$ for each $\alpha\in T$,  and  
 
 \item  if $\lambda \in [T]$ then the sequence $X_{\lambda\restriction 0}, X_{\lambda\restriction 1}, \ldots$  converges to a point  in $X$.
\end{enumerate}
A space  is absolute $F_{\sigma\delta}$ if and only if it has a Sierpi\'{n}ski stratification {\cite[Th\'{e}or\`{e}me]{si}}.

\begin{ut}An almost zero-dimensional space $X$ is an Erd\H{o}s space factor if and only if $X$ has a Sierpi\'{n}ski stratification of C-sets.\end{ut}

\begin{proof}Suppose that $X$ is an Erd\H{o}s space factor.  Then $X$ is homeomorphic to a closed subset of $\mathfrak E$. Since $\mathfrak E\simeq \mathbb Q ^\omega\times \mathfrak E_{\mathrm{c}}$ \cite[Proposition 9.1]{erd}, we may assume that $X$ is a closed subset of $\mathbb Q ^\omega\times \mathfrak E_{\mathrm{c}}$.  Let $(A_{\alpha})_{\alpha\in S} $ be  the obvious Sierpi\'{n}ski stratification of $\mathbb Q ^\omega$ in which  $S$ is a tree over $\mathbb Q$. Let $d$ be a complete metric for $\mathfrak E_{\mathrm{c}}$. For each $n<\omega$ let $\{B^n_i:i<\omega\}$ be a C-set covering  of $\mathfrak E_{\mathrm{c}}$ such that $\diam(B^n_i)<1/n$ in the metric $d$.  Let $B_{\varnothing}=\mathfrak E_{\mathrm{c}}$. For each non-empty $\beta\in \omega^{<\omega}$ define $B_{\beta}=\bigcap\{B^n_{\beta(n)}:n<\dom(\beta)\}$. Then $T:=\{\beta\in \omega^{<\omega}:B_{\beta}\neq\varnothing\}$ is a tree over $\omega$, and  by completeness of $(\mathfrak E_{\mathrm{c}},d)$ we have that $(B_{\beta})_{\beta\in T}$ is a Sierpi\'{n}ski stratification of $\mathfrak E_{\mathrm{c}}$.  If $\alpha\in \mathbb Q ^{<\omega}$,  $\beta\in \omega^{<\omega}$, and $n=\dom(\alpha)=\dom(\beta)$, then we define $\alpha*\beta =\langle \langle \alpha(0),\beta(0)\rangle, \ldots, \langle \alpha(n-1),\beta(n-1)\rangle\rangle.$ Note that $$S * T := \{\alpha*\beta : \alpha \in S,\; \beta \in  T, \text{ and } \dom(\alpha) = \dom(\beta)\}$$  is a tree over $\mathbb Q \times \omega$,   $(A_{\alpha}\times B_{\beta})_{\alpha*\beta\in S*T}$ is a Sierpi\'{n}ski stratification of $\mathbb Q ^\omega\times \mathfrak E_{\mathrm{c}}$, and each $A_{\alpha}\times B_{\beta}$ is a C-set in $\mathbb Q ^\omega\times \mathfrak E_{\mathrm{c}}$. Then $((A_\alpha\times B_{\beta}) \cap X)_{{\alpha*\beta\in S*T}}$ is a Sierpi\'{n}ski stratification of $X$ consisting of C-sets in $X$.

Now suppose that $(A_\alpha)_{\alpha\in T}$ is a Sierpi\'{n}ski stratification of $X$ where every $A_\alpha$ is a C-set in $X$. For each $\alpha\in T$ write $A_\alpha=\bigcap \{C^\alpha_n:n<\omega\}$ where each $C^\alpha_n$ is clopen in $X$.  Let $\{B_i:i<\omega\}$ be a neighborhood basis of C-sets for $X$, and for each $i<\omega$ write  $B_i=\bigcap \{D_{ij}:j<\omega\}$ where each $D_{ij}$ is clopen in $X$.  The  topology $\mathfrak T$ that is generated by the sub-basis 
$$\{C^\alpha_n,X\setminus C^\alpha_n:\alpha\in T,\; n< \omega\}\cup \{D_{ij},X\setminus D_{ij}:i,j<\omega\}$$ is easily seen to be a second countable, regular, zero-dimensional topology on $X$. It witnesses the almost zero-dimensionality of $X$ because $\mathfrak T$ is coarser than the original topology of $X$ and every $B_i$ is $\mathfrak T$-closed. Further, every $A_\alpha$ is $\mathfrak T$-closed and $(A_\alpha)_{\alpha\in T}$ is a Sierpi\'{n}ski stratification of $(X,\mathfrak T)$. By Sierpi\'{n}ski's theorem   $(X,\mathfrak T)$ is  an $F_{\sigma\delta}$-space. By Proposition 1,  $X$ is an Erd\H{o}s space factor. \end{proof}

\section{Countable unions of C-set $\mathfrak E$-factors}

The power of Theorem 1 is demonstrated in the proof of the following. 
\begin{ut}If $X$ is an almost zero-dimensional space which is the union of countably many C-set Erd\H{o}s space factors, then $X$ is an Erd\H{o}s space factor. \end{ut}

\begin{proof}Suppose that $X$ is almost zero-dimensional and $X=\bigcup \{A_n:n<\omega\}$ where each $A_n$ is both an  Erd\H{o}s space factor and a C-set in $X$.  By Theorem 1, for each $n<\omega$ there is a Sierpi\'{n}ski stratification  $(B^n_\alpha)_{\alpha\in T_n}$  of $A_n$ such that each $B^n_\alpha$ is a C-set in $A_n$.   Define a tree $$T=\bigcup_{n<\omega}\{\langle n\rangle^\frown \alpha:\alpha\in T_n\}.$$ Put $X_{\varnothing}=X$ and  $X_{\langle n\rangle^\frown \alpha}=B^n_\alpha$ for each $n<\omega$ and $\alpha\in T_n$. By \cite[Corollary 4.20]{erd} each $B^n_{\alpha}$ is a C-set in $X$. Thus $(X_\alpha)_{\alpha\in T}$ is a Sierpi\'{n}ski stratification of $X$ consisting of C-sets in $X$.  By Theorem 1,  $X$ is an Erd\H{o}s space factor. \end{proof}

By Proposition 1 and Theorem 2 we have:

\begin{uc} Let $X$ be an almost zero-dimensional space. If $X$  can be written as the union of countably many  C-sets which admit closed embeddings into $\mathfrak E$, then $X$ has a closed embedding in $\mathfrak E$.\end{uc}

We will now present an application of Theorem 2. For any space $X$ and $n<\omega$ we let  $\mathcal F_n(X)$ denote the set of all non-empty subsets of $X$ of cardinality $\leq n$.  The set of all non-empty finite subsets of $X$ is denoted $$\mathcal F(X)=\bigcup_{n<\omega} \mathcal F_n(X).$$ The Vietoris topology on $\mathcal F(X)$ has a basis  of open sets of the form $$\textstyle \langle U_0,\ldots,U_{k-1}\rangle=\{F\in \mathcal F(X): F\subset \bigcup_{i<k} U_i\text{ and }F\cap U_i\neq \varnothing\text{ for each }i<k\},$$ 
where $k<\omega$ and $U_0,\ldots,U_{k-1}$ are non-empty open subsets of $X$.
\begin{uc}If $X$ is an Erd\H{o}s space factor then so is $\mathcal F(X)$. \end{uc}

\begin{proof}Suppose that $X$ is an Erd\H{o}s space factor. Then $\mathcal F(X)$ is almost zero-dimensional by \cite[Proposition 2.2]{zar}.  By \cite[Corollary 5.2]{zar} each  $\mathcal F_n(X)$ is an Erd\H{o}s space factor. And each $\mathcal F_n(X)$ is a C-set in $\mathcal F(X)$.  Indeed, if $F\in \mathcal F(X)\setminus \mathcal F_n(X)$ then we can partition $X$ into $m=|F|$  pairwise disjoint non-empty clopen sets $C_0,\ldots, C_{m-1}$ such that $F\in \langle C_0,\ldots, C_{m-1}\rangle$.  Since $\langle C_0,\ldots, C_{m-1}\rangle$ is clopen in $\mathcal F(X)$ and  $\langle C_0,\ldots, C_{m-1}\rangle\subset \mathcal F(X)\setminus \mathcal F_n(X)$,  this shows that $\mathcal F_n(X)$ is a C-set in $\mathcal F(X)$.  We conclude that $\mathcal F(X)$ is an almost zero-dimensional countable union of C-set Erd\H{o}s space factors. By Theorem 2 $\mathcal F(X)$ is an Erd\H{o}s space factor.  \end{proof}

By Theorem 2 and Proposition 2 we have:

\begin{uc}If $X$ is  almost zero-dimensional space that is a countable union of complete C-sets, then $X$ is an Erd\H{o}s space factor.  \end{uc}

We will now apply Corollary 5 to the main example in \cite{dij}.  If $X$ is an almost zero-dimensional space and $A$ is a C-set in $X$, then define $$T(X,A)=\bigcup_{n<\omega} (X\setminus A)^n\times A.$$ Let $\pi : T(X, A) \to X$ be defined by $\pi(x_0,\ldots ,x_{n}) = x_0$.
Consider the collection $\mathcal B$ of subsets of $T(X, A)$ that consists of all sets of the form
$O_0 \times\ldots \times O_{n-1} \times \pi^{-1}(O_{n})$, where $n <\omega$, each $O_i$ is an open subset of $X$, and $O_i\subset X\setminus A$ for each $i<n$. By \cite[Claim 7]{dij}, $\mathcal B$ forms a basis for an almost zero-dimensional topology on  $T(X,A)$.

\begin{uc}Dijkstra's homogeneous almost zero-dimensional space  is a countable union of complete C-sets  and is therefore an Erd\H{o}s space factor.
\end{uc}

\begin{proof}The example $T(\tilde E,E')$ in \cite{dij} is generated by a particular complete almost zero-dimensional space $\tilde E\simeq \mathfrak E_{\mathrm{c}}$  and a C-set $E'$ in $\tilde E$. We claim that for $X=\tilde E$ and $A=E'$, the space $T(X,A)$ is a countable union of complete C-sets. To prove this, for each $n<\omega$ put  $$T_n(X,A)=\bigcup_{k\leq n} (X\setminus A)^k\times A.$$  Write $A$ as an intersection of clopen subsets of $X$;  $A=\bigcap\{C_i:i<\omega\}$.  Then $$T_n(X,A)=\bigcap_{i<\omega}\bigcup _{k\leq n}(X\setminus A)^k\times \pi^{-1}(C_i).$$ For each $i<\omega$ note that $\bigcup _{k\leq n}(X\setminus A)^k\times \pi^{-1}(C_i)$ is clopen in $T(X,A)$ because it is a union of basic open sets and its  complement is the basic open set $(X\setminus A)^n\times \pi^{-1}(X\setminus C_i)$.  Therefore $T_n(X,A)$ is a C-set in $T(X,A)$.  Finally,  each   $T_n(X,A)$ is  complete by \cite[Claim 6]{dij}. \end{proof}

\section{$\mathfrak E$ is unstable}

To prove the main result in this section, we will need the following lemma.
\begin{ul}Let $X$ be an almost zero-dimensional space, and suppose that $(A_\alpha)_{\alpha\in T}$ is a Sierpi\'{n}ski stratification of  $X$ consisting of C-sets in $X$.  Then there exists  a Sierpi\'{n}ski stratification $(B_{\beta})_{\beta\in S}$ of $X$ such that every $B_{\beta}$ is a non-empty  C-set in $X$, and $\diam(B_{\beta})<1/\dom(\beta)$.\end{ul}

\begin{proof}For each $n<\omega$ let   $\{C^n_i:i<\omega\}$ be a C-set covering of $X$ such that $\diam(C^n_i)<1/(n+1)$ for all $i<\omega$. For every $\beta=\alpha*\gamma\in T*\omega^{<\omega}$ (where the $*$ operation is defined as in the proof of Theorem 1), define $$B_\beta=A_{\alpha}\cap \bigcap_{n<\dom(\alpha)} C^n_{\gamma(n)}.$$ Let $S=\{\beta\in T*\omega^{<\omega} :B_{\beta}\neq\varnothing\}$. Then $(B_{\beta})_{\beta\in S}$ is as desired.\end{proof}

\begin{ut}Every first category  Erd\H{o}s space factor contains a neighborhood which is covered by countably many nowhere dense C-sets.\end{ut}

\begin{proof}Let $X$ be an Erd\H{o}s space factor.   Suppose that no neighborhood in $X$ can be covered by countably many nowhere dense C-sets of $X$.  We will show that $X$ is not first category.  To that end, let $\{X_i:i<\omega\}$ be any (countable) collection of closed nowhere dense subsets of $X$. We will show $X\neq \bigcup \{X_i:i<\omega\}$.

 By Theorem 1 there is a Sierpi\'{n}ski stratification $(A_\alpha)_{\alpha\in T}$ of $X$ such that every $A_\alpha$ is a C-set in $X$. By Lemma 7 we may assume that the $A_\alpha$'s are non-empty and $\diam(A_{\alpha})<1/\dom(\alpha)$ for each $\alpha\in T$. We will now inductively define a sequence $(\alpha^i)\in T^\omega$ and integers $N_0<N_1<\ldots$ such that for every $i<\omega$:
\begin{itemize}\renewcommand{\labelitemi}{\scalebox{.5}{{$\blacksquare$}}}
\item $\dom(\alpha^i)=N_i$,
\item  $\alpha^{i+1}\restriction N_i=\alpha^i$,
\item $A_{\alpha^i}$ has non-empty interior in $X$, and
\item $A_{\alpha^i}\cap X_i=\varnothing$.
\end{itemize}
To begin the construction,   let $x\in X\setminus X_0$.  Choose $N_0\geq 1$ such that $$2/N_0<\varepsilon:=d(x,X_0).$$  Since $\{A_\alpha:\alpha\in T,\; \dom(\alpha)=N_0\text{, and }A_\alpha\cap B(x,\varepsilon/2)\neq\varnothing\}$ is a countable C-set covering of the neighborhood $B(x,\varepsilon/2)$, there exists $\alpha^0\in T$ such that $\dom(\alpha^0)=N_0$, $A_{\alpha^0}$ contains a non-empty open subset of $X$, and $A_{\alpha^0}\cap  B(x,\varepsilon/2)\neq\varnothing$. The last condition implies $A_{\alpha^0}\cap X_0=\varnothing$ because if $y\in A_{\alpha^0}$ then $$d(x,y)\leq  \varepsilon/2+ \diam(A_{\alpha^0})<\varepsilon/2+1/\dom(\alpha^0)=\varepsilon/2+1/N_0<\varepsilon.$$

Now suppose $\alpha^i\in T$ and $N_i$ have been appropriately defined for a given $i<\omega$.   Since  $A_{\alpha^i}$ contains a non-empty open subset of $X$, and $X_{i+1}$ is nowhere dense in $X$, there exists $x\in X$ and $\varepsilon>0$ such that $B(x,\varepsilon)\subset  A_{\alpha^i}\setminus X_{i+1}$. Choose $N_{i+1}>N_i$ such that $2/N_{i+1}<\varepsilon$. Observe  that $$\{A_\alpha:\alpha\in T,\;\dom(\alpha)=N_{i+1},\; \alpha\restriction N_i=\alpha^i,\text{ and }A_\alpha\cap B(x,\varepsilon/2)\neq\varnothing\}$$ covers $B(x,\varepsilon/2)$. So there exists $\alpha^{i+1}\in T$ such that $\dom(\alpha^{i+1})=N_{i+1}$, $\alpha^{i+1}\restriction N_i=\alpha^i$, $A_{\alpha^{i+1}}$ has non-empty interior in  $X$, and $A_{\alpha^{i+1}}\cap  B(x,\varepsilon/2)\neq\varnothing$ (hence $A_{\alpha^{i+1}}\cap X_{i+1}=\varnothing$). Thus the construction can be continued.

Finally, let $$\lambda=\bigcup_{i<\omega} \alpha^i\in [T].$$ By the convergence property (4) in the definition of a Sierpi\'{n}ski stratification, the sets $A_{\lambda\restriction 0},A_{\lambda\restriction 1},\ldots$ converge to a point $x_{\lambda}\in X$. Then $$x_{\lambda}\in \bigcap_{n<\omega} A_{\lambda \restriction n} =\bigcap_{i<\omega} A_{\alpha^i}\subset X\setminus \bigcup_{i<\omega} X_i.$$ Therefore $X\neq \bigcup \{X_i:i<\omega\}$. \end{proof}

We are now ready to present counterexamples to Question 1.  

Our first example is elementary and will require two basic lemmas concerning the topologies of $\ell^2$ and $\mathfrak E_{\mathrm{c}}$.

\begin{ul}Let $x,x_0,x_1,\ldots\in \ell^2$.  Suppose $x_n\to  x$ in  $\mathbb R ^\omega$ and $\|x_n\|<\|x\|+ 1/n$ for each $n<\omega$. Then $x_n\to x$ in $\ell^2$.\end{ul}

\begin{proof}
By  coordinate-wise convergence of $(x_n)$ and the lemma in \cite{rob}, it suffices to show $\|x_n\|\to \|x\|$.  To that end, let $\varepsilon>0$.  Since $\{y\in \mathbb R ^\omega:\|y\|>\|x\|-\varepsilon\}$ contains $x$ and is open in $\mathbb R ^\omega$,  there exists $N$ such that $1/N<\varepsilon$ and $\|x_n\|>\|x\|-\varepsilon$ for all $n\geq N$.  Then $\|x\|-\varepsilon<\|x_n\|<\|x\|+\varepsilon$ for all $n\geq N$. Thus $\|x_n\|\to \|x\|$. \end{proof}

\begin{ul}$\mathfrak E_{\mathrm{c}}$ has a basis of neighborhoods of the form $\{x\in C:\|x\|\leq q\}$ where $C$ is clopen in $\mathbb P ^\omega$ and $q\in \mathbb Q$. \end{ul}

\begin{proof}Let $y\in \mathfrak E_{\mathrm{c}}$ and let $U$ be any open subset of $\mathfrak E_{\mathrm{c}}$ with $y\in U$. There exists $(q_n)\in \mathbb Q^\omega$ such that  $\|y\|<q_n<\|y\|+1/n$ for every $n<\omega$. There is also a sequence $C_0\supset C_1\supset\ldots$ of clopen subsets of $\mathbb P ^\omega$ which converges to $y$ in the topology of $\mathbb P ^\omega$.  If $x_n\in \{x\in C_n:\|x\|\leq q_n\}$ for every $n<\omega$, then $x_n\to y$ by Lemma 9.  Thus there exists $n<\omega$ such that  $\{x\in C_n:\|x\|\leq q_n\}\subset U$. Clearly $\{x\in C_n:\|x\|\leq q_n\}$ is an $\mathfrak E_{\mathrm{c}}$-neighborhood of $y$.\end{proof}

Let $\mathfrak E'=\{x\in \ell^2:x_n\in \mathbb Q+\sqrt{2}\text{ for all }n<\omega\}$.

\begin{ue}$X:=\mathfrak E'\cup \{x\in \mathfrak E_{\mathrm{c}}:\|x\|\in \mathbb Q\}$ is an almost zero-dimensional $F_{\sigma\delta}$-space which is not an Erd\H{o}s space factor.\end{ue}

\begin{proof}Note that $\mathfrak E'\simeq \mathfrak E$ and $\{x\in \mathfrak E_{\mathrm{c}}:\|x\|=q\}$ is nowhere dense in $X$ for each $q\in \mathbb Q$.  So $X$ is first category.  Further, $X$ is the union of two $F_{\sigma\delta}$-subsets of $\mathfrak E_{\mathrm{c}}$ and is therefore almost zero-dimensional and absolute $F_{\sigma\delta}$. In order to apply Theorem 8, we need to show that no neighborhood in $X$ can be covered by countably many nowhere dense C-sets of $X$.  Let $A$ be any neighborhood in $X$. By Lemma 10 we may assume that  $A=\{x\in C:\|x\|\leq q\}$ where $C$ is clopen in $\mathbb P ^\omega$ and $q\in \mathbb Q$.   For a contradiction, suppose that $A\subset \bigcup \{A_n:n<\omega\}$ where each $A_n$ is a nowhere dense C-set in $X$. Note that $\{x\in C:\|x\|=q\}$ is complete, and its topology as a subspace of $X$ is the same as the topology it inherits from $\mathbb P ^\omega$. By Baire's theorem there is a clopen set $B\subset C$ and $n<\omega$ such that  $\varnothing\neq \{x\in B:\|x\|=q\}\subset A_n$. Then  the open set $\{x\in B:\|x\|<q\}$ is also non-empty. Since $A_n$ is nowhere dense, there exists $x\in X\cap B\setminus A_n$ such that $\|x\|<q$.  Let $O$ be a clopen subset of $X$ such that $x\in O$ and $O\cap A_n=\varnothing$.  Then $O\cap A\cap \mathfrak E'$ is a non-empty bounded clopen subset of $\mathfrak E'$. This contradicts a key property of the Erd\H{o}s space \cite{dims}, namely that $\mathfrak E'\cup \{\infty\}$ is connected.\end{proof}

Our next two examples are from complex dynamics. Define $f(z)=\exp(z)-1$ for each $z\in \mathbb C$. The Julia set $J(f)$ is a Cantor bouquet of rays in the complex plane \cite{aa}, and has  a natural endpoint set $E(f)$ which is almost zero-dimensional and complete.  In fact, $E(f)$ is  homeomorphic to $\mathfrak E_{\mathrm{c}}$ \cite{31}. We let $$f^n=\underbrace{f\circ f\circ\ldots\circ f}_\text{$n$ times}$$ denote the $n$-fold composition of $f$.

Recall that a space $X$ is \textbf{nowhere $\sigma$-complete} if no neighborhood in $X$ can be written as a countable union of complete subspaces.  This is equivalent to saying that no neighborhood in $X$ is absolute $G_{\delta\sigma}$ (i.e.\ $X$ is nowhere $G_{\delta\sigma}$). 

\begin{ue}There is an almost zero-dimensional $F_{\sigma\delta}$-space that is nowhere $G_{\delta\sigma}$ and is not an Erd\H{o}s space factor.\end{ue}

\begin{proof} The escaping endpoint set $\dot E(f):=\{z\in E(f):f^{n}(z)\to\infty\}$ is almost zero-dimensional and first category \cite{lip}, $F_{\sigma\delta}$ and nowhere $G_{\delta\sigma}$ \cite{lip5}, and no neighborhood in $\dot E(f)$  can be covered by countably many nowhere dense C-sets of $\dot E(f)$ \cite[Remark 5.3]{lip3}. By Theorem 8, $\dot E(f)$ is not an Erd\H{o}s space factor.\end{proof}

By contrast, a space $X$ is \textbf{strongly $\sigma$-complete} if $X$ can be written as a countable union of closed complete subspaces. 

\begin{ue}There is a strongly $\sigma$-complete almost zero-dimensional space  which is not an Erd\H{o}s space factor.\end{ue}

\begin{proof}Consider the set  $\hat E(f):=\{z\in E(f):\overline{\{f^n(z):n<\omega\}}\neq J(f)\}$ consisting of all endpoints whose (forward) orbits are not dense in $J(f)$. By \cite[Lemma 1]{bak},  $\hat E(f)$ is a first category  $F_{\sigma}$-subset of  $E(f)$.  In particular, $\hat E(f)$ is strongly $\sigma$-complete and absolute $F_{\sigma\delta}$. Clearly $\dot E(f)\subset \hat E(f)$, and $\dot E(f)$ is dense in $E(f)$ by Montel's theorem. Thus no neighborhood in $\hat E(f)$ can be covered by countably many nowhere dense C-sets of $\hat E(f)$.  By Theorem 8, $\hat E(f)$ is not an Erd\H{o}s space factor.\end{proof}

\begin{ur}In light of Corollary 5,  $\hat E(f)$ is an example of a strongly $\sigma$-complete almost zero-dimensional space  which cannot be written as a countable union of complete C-sets.\end{ur}

\begin{ur}van Mill proved that  $\mathbb Q\times \mathbb P$ is the  unique  zero-dimensional space which is strongly $\sigma$-complete, nowhere $\sigma$-compact, and nowhere complete  \cite{van}. There is no such classification of almost zero-dimensional spaces, as $\hat E(f)$ is strongly $\sigma$-complete,  nowhere $\sigma$-compact, nowhere complete, and is not a $\mathbb Q$-product. 
\end{ur}


\end{document}